\documentclass[10pt]{amsart}
\usepackage{graphicx}
\usepackage{amscd}
\usepackage{amsmath}
\usepackage{amsthm}
\usepackage{amsfonts}
\usepackage{amssymb}
\usepackage{mathrsfs}
\usepackage{bm}
\usepackage{enumerate}
\usepackage{amsrefs}
\usepackage{xcolor}
\usepackage[colorlinks, citecolor=blue, linkcolor=red, pdfstartview=FitB]{hyperref}
\usepackage{soul} 

\newcommand{\bB}{{\mathbb{B}}}
\newcommand{\bC}{{\mathbb{C}}}
\newcommand{\bD}{{\mathbb{D}}}

\newcommand{\bF}{{\mathbb{F}}}

\newcommand{\bN}{{\mathbb{N}}}

\newcommand{\bT}{{\mathbb{T}}}

  \newcommand{\A}{{\mathcal{A}}}
  \newcommand{\B}{{\mathcal{B}}}

  \newcommand{\E}{{\mathcal{E}}}
  \newcommand{\F}{{\mathcal{F}}}
  \newcommand{\G}{{\mathcal{G}}}
\renewcommand{\H}{{\mathcal{H}}}

  \newcommand{\K}{{\mathcal{K}}}  
\renewcommand{\L}{{\mathcal{L}}}
  \newcommand{\M}{{\mathcal{M}}}
  
\renewcommand{\O}{{\mathcal{O}}}

  \newcommand{\R}{{\mathcal{R}}}
\renewcommand{\S}{{\mathcal{S}}}

  \newcommand{\V}{{\mathcal{V}}}
  \newcommand{\W}{{\mathcal{W}}}

\newcommand{\fA}{{\mathfrak{A}}}

\newcommand{\fe}{{\mathfrak{e}}}
\newcommand{\fF}{{\mathfrak{F}}}

\newcommand{\fH}{{\mathfrak{H}}}

\newcommand{\fK}{{\mathfrak{K}}}
\newcommand{\fL}{{\mathfrak{L}}}
\newcommand{\fM}{{\mathfrak{M}}}
\newcommand{\fN}{{\mathfrak{N}}}

\newcommand{\fq}{{\mathfrak{q}}}
\newcommand{\fR}{{\mathfrak{R}}}

\newcommand{\fT}{{\mathfrak{T}}}

\newcommand{\rA}{\mathrm{A}}

\newcommand{\rC}{\mathrm{C}}

\renewcommand{\phi}{\varphi}

\newcommand{\upchi}{{\raise.35ex\hbox{$\chi$}}}

\newcommand{\lip}{\langle}
\newcommand{\rip}{\rangle}
\newcommand{\ip}[1]{\lip #1 \rip}

\newcommand{\ol}{\overline}

\newcommand{\bksl}{\backslash}

\newcommand{\Cstar}{\mathrm{C}^*}

\newcommand{\qand}{\quad\text{and}\quad}

\newcommand{\id}{\operatorname{id}}

\newcommand{\Hen}{\operatorname{Hen}}

\newtheorem{lemma}{Lemma}[section]
\newtheorem{theorem}[lemma]{Theorem}
\newtheorem{proposition}[lemma]{Proposition}
\newtheorem{corollary}[lemma]{Corollary}
\newtheorem{theoremx}{Theorem}

\theoremstyle{definition}

\newtheorem{example}{Example}

\newcommand{\wh}[1]{\widehat{#1}}
\newcommand{\tripnorm}[1]{{\left\vert\kern-0.25ex\left\vert\kern-0.25ex\left\vert #1 
    \right\vert\kern-0.25ex\right\vert\kern-0.25ex\right\vert}}

\begin{document}

\author{Rapha\"el Clou\^atre}
\email{raphael.clouatre@umanitoba.ca\vspace{-2ex}}
\author{Robert T. W. Martin}
\email{robert.martin@umanitoba.ca\vspace{-2ex}}
\author{Edward J. Timko}
\email{edward.timko@umanitoba.ca\vspace{-2ex}}
\address{Department of Mathematics, University of Manitoba, Winnipeg, Manitoba, Canada R3T 2N2}

\thanks{
	R.C. was partially supported by an NSERC Discovery Grant. R.T.W.M. was partially supported by an NSERC Discovery Grant. E.J.T. was partially supported by a PIMS postdoctoral fellowship.
}

\title[Analytic functionals for the nc disc algebra]{Analytic functionals for the non-commutative disc algebra}


\begin{abstract}
The main objects of study in this paper are those functionals that are analytic in the sense that they annihilate the non-commutative disc algebra. In the classical univariate case, a theorem of F. and M. Riesz implies that such functionals must be given as integration against an absolutely continuous measure on the circle. We develop generalizations of this result to the multivariate non-commutative setting, upon reinterpreting the classical result. In one direction, we show that the GNS representation naturally associated to an analytic functional on the Cuntz algebra cannot have any singular summand. Following a different interpretation, we seek weak-$*$ continuous extensions of analytic functionals on the free disc operator system. In contrast with the classical setting, such extensions do not always exist, and we identify the obstruction precisely in terms of the so-called universal structure projection. We also apply our ideas to commutative algebras of multipliers on some complete Nevanlinna--Pick function spaces.
\end{abstract}

\maketitle

\section{Introduction} 
The classical analytic measure theorem of F. and M. Riesz is a cornerstone of  harmonic analysis on the complex  open  unit disc $\bD$. It implies that the Borel measures on the unit circle $\bT$ that annihilate the disc algebra $\rA(\bD)$ -- the so-called \emph{analytic} measures -- must all be absolutely continuous with respect to Lebesgue measure $\lambda$ \cite[page 47]{hoffman1969}. This result has had an impact in operator theory as well. Indeed, it is a key tool in characterizing certain Hilbert space contractions that admit a rich functional calculus \cite{nagy2010},\cite{bercovici1988}. 

Another formulation of the F. and M. Riesz theorem can be given in terms of the $\rC^*$-algebra $\rC(\bT)$ of continuous functions viewed inside the abelian von Neumann algebra $L^\infty(\bT,\lambda)$: any bounded linear functional on $\rC(\bT)$ that annihilates $\rA(\bD)$ must extend weak-$*$ continuously to $L^\infty(\bT,\lambda)$. This rephrasing naturally prompts potential generalizations to more abstract operator algebraic settings.
The current paper is a contribution to this program. 

 Let us describe our framework more precisely.  Let $d\geq 1$ be an integer and let $L=(L_1,\ldots,L_d)$ denote the left free shift acting on the Fock space $\fF^2_d$ over $\bC^d$; this is given by the left regular representation of the free monoid $\bF_d^+$. The unital norm closed operator algebra generated by $L_1,\ldots,L_d$ is Popescu's algebra  \cite{popescu1991}, which we denote by $\A_d$. We will be working with various other objects generated by $\A_d$ inside of $B(\fF^2_d)$. Indeed, we consider both the norm closed and the weak-$*$ closed operator system generated by $\A_d$, which we denote respectively by $\S_d$ and $\V_d$.  We also define $\fT_d=\rC^*(\A_d)\subset B(\fF^2_d)$; this is the Cuntz--Toeplitz algebra. This $\rC^*$-algebra contains the ideal of compact operators, and the corresponding quotient is the Cuntz algebra $\O_d$. Let $q:\fT_d\to \O_d$ denote the quotient map. It is well know that $q$ is completely isometric on $\S_d$ \cite{popescu1996}.
 
 Recall now that $\O_1\cong \rC(\bT)$, and under this identification we have that $q(\A_1)$ coincides with $\rA(\bD)$. For this reason, $\A_d$ is often referred to as the \emph{non-commutative disc algebra}, and $\S_d$ as the \emph{free disk operator system}. Naturally, $\A_d$ and $q(\A_d)$ will then serve as our multivariate replacements for $\rA(\bD)$. When $d=1$, we have that $q(\S_1)=\O_1$. Thus, both $\O_d$ and $\S_d\cong q(\S_d)$ are sensible multivariate replacements for $\rC(\bT)$. 
 
 %
Keeping the aforementioned analogies in mind, the driving force behind our work, generally speaking, is to understand the structure of those bounded functionals $\tau$ on $\O_d$ or $\S_d$ that annihilate the copy of $\A_d$. From this point of view, what we are after is a multivariate non-commutative F. and M. Riesz theorem. In doing so, we further reinforce the philosophy that questions surrounding $\A_d$ should be thought of as non-commutative Hardy space theory.

It should be acknowledged that similar investigations have been undertaken in the past, see for instance \cite{exel1990},\cite{BL2007},\cite{BL2007p},\cite{ueda2009} and the references therein.  Our framework is readily seen to be disjoint from that of these previous works when $d>1$. Our study is close in spirit to that of \cite{KY2014} (especially Corollary 4.2 therein), but our results are of a different nature altogether as we are interested in functionals defined on a self-adjoint object containing $\A_d$.  Thus, there is no obvious overlap between our work and \cite{KY2014}.

Our efforts aim to show that the condition that $\tau$ annihilates $\A_d$ should guarantee that $\tau$ is ``absolutely continuous". A major subtlelty inherent to this question is to meaningfully interpret absolute continuity in our context. This will be accomplished via the classification theory of $*$-representations of $\O_d$ initiated in \cite{davidson2001},\cite{davidson2005} that culminated with the deep results of \cite{kennedy2013}. Therein, notions of absolute continuity and singularity were introduced for a $*$-representation of the Cuntz algebra, and some of our main results will be given in those terms. Notably, these two notions are not complementary when $d>1$; indeed, in the multivariate setting there emerges a third type for representations, called dilation type, whose presence complicates the theory significantly. It is this difference between the classical and multivariate settings that accounts for much of the new behavior that we record in this paper. Section \ref{S:Prelims} contains additional details on these matters and further background material. 

Care must be taken to interpret  the statement of the classical F. and M. Riesz theorem in a way that is susceptible to generalization. This is done in Section \ref{S:classFM}, which serves as a springboard for our original contributions to follow.

In Section \ref{S:FMRCNS}, we develop our first generalization of the classical F. and M. Riesz theorem. We interpret the latter as producing a ``non-singular" representation for any analytic functional on $\O_1$. We show that this paradigm is still valid in the non-commutative multivariate setting of $\O_d$. Our main tool for this purpose is the following. Let $\tau$ be a bounded linear functional on $\O_d$. Let $\pi:\O_d\to B(\fH)$ be a unital $*$-representation and $\xi,\eta\in \fH$ be vectors with norm $\|\tau\|^{1/2}$. Let $\fM=\ol{\pi(q(\A_d))\xi}$ and let $\rho:\A_d\to B(\fM)$ be the corresponding representation induced  by $\pi$. We say that the triple $(\pi,\xi,\eta)$ is a \emph{Riesz representation} for $\tau$ if $\xi$ is cyclic for $\pi$, if $\rho$ is unitarily equivalent to the identity representation of $\A_d$, and if
\[
\tau(t)=\langle \pi(t)\xi,\eta\rangle, \quad t\in \O_d.
\]
The following is our first main result, appearing as Theorem \ref{T:MCEFMR} below.
\begin{theoremx}\label{T:mainA}
Let $\tau$ be a bounded linear functional on the Cuntz algebra $\O_d$. If $\tau$ annihilates $\A_d$, then $\tau$ admits an essentially unique Riesz representation. Furthermore, this Riesz representation has no singular summand.
\end{theoremx}
The classical theorem gives in fact a bit more: in that case, the $*$-representation of $\O_1$ given by the bilateral shift on $L^2(\bT,\lambda)$ can serve as a Riesz representation for \emph{all} functionals annihilating $\A_1$. Interestingly, there is no such universal Riesz representation in the multivariate setting, as shown in Corollary \ref{C:NoOneRep}.


Section \ref{S:FMRnAC} contains our second main result. Therein, we adopt the point of view that the classical F. and M. Riesz theorem yields the existence of weak-$*$ continuous extensions to $L^\infty(\bT,\lambda)$ for analytic functionals on $\rC(\bT)$. As explained in Section \ref{S:classFM}, this is equivalent to the existence of weak-$*$ continuous extensions to $\V_1$ for analytic functionals on $\S_1$. Whether or not an analogous extension statement holds for the inclusion $\S_d\subset \V_d$ is the focus of this section. As a first indication that this may be the case, given an analytic functional $\tau$ on $\O_d$, we manage to show that the restriction of the absolute value $|\tau|$ to $\S_d$ extends weak-$*$ continuously to $\V_d$ (Theorem  \ref{T:TdAbsValAC}). Of course, it is elementary to verify, in the classical univariate setting, that this statement about the absolute value is equivalent to the F. and M. Riesz theorem. Unfortunately, the situation is more complicated in general. Indeed, when $d\geq 2$, we exhibit an example of an analytic functional on $\O_d$ that annihilates $\A_d$ yet does not extend weak-$*$ continuously to $\V_d$ (see Example \ref{E:NoACForFree}).

As mentioned previously, a functional annihilating $\A_d$ admits a Riesz representation without singular summands. If, in addition, we knew that the representation contained no summand of dilation type, then necessarily $\tau$ would extend weak-$*$ continuously to $\V_d$. Thus, the presence of a dilation type summand is central to the issue investigated in this section. This is further reinforced by the pathological construction in Example \ref{E:NoACForFree}, which involves a representation of dilation type. In light of these remarks, one may venture the guess that the lack of weak-$*$ continuous extensions is equivalent to the presence of a dilation type summand in the Riesz representation of the functional. However, we show in Corollary \ref{C:dilationextends} that this is not the case, so  that there is seemingly no obvious relationship between the type of the Riesz representation of an analytic functional and the possibility of extending it weak-$*$ continuously. The main result of this section (Theorem \ref{T:ACComm})  clarifies the picture and identifies the obstruction to the existence of such extensions precisely in terms of the universal structure projection $\fq$ of $\A_d^{**}$ (see Subsection \ref{SS:ncda} for details on this projection).

\begin{theoremx}\label{T:mainB}
Let $\tau$ be a bounded linear functional on $\S_d$ that annihilates $\A_d$. Then, $\tau$ extends weak$^*$-continuously to $\V_d$ if and only if it also annihilates $a^*\fq-\fq a^*$ for each $a\in\A_d$.
\end{theoremx}

In accordance with the classical F. and M. Riesz theorem, the condition above is automatically satisfied in the univariate setting. Indeed, since $\A_1$ is commutative we have that $a^*\fq-\fq a^*=0$ for every $a\in\A_1$.


In the last section, Section \ref{S:Comm}, we adapt some of our ideas to the setting of commutative algebras of multipliers for some Hilbert function spaces. Exploiting the recent developments of \cite{DavHar2020}, we investigate the existence of weak-$*$ continuous extensions for functionals that are analytic in an appropriate sense (Theorem \ref{T:CommFMR}). As an application, we manage to prove in Corollary \ref{C:HenA=HenTA} a certain symmetry property of the set of so-called Henkin functionals. Such functionals have been a topic of recent interest, as they serve as a bridge between function theory and multivariate operator theory (see \cite{CD2016duality},\cite{CD2016abscont},\cite{BicHarMcC2018},\cite{DavHar2020} and the references therein). For most Hilbert function spaces, their exact nature is not yet fully understood, and our result sheds light on this question.

\section{Preliminaries}\label{S:Prelims} 

Throughout the paper, $\fH$ always denotes a Hilbert space and $B(\fH)$ denotes the $\rC^*$-algebra of bounded linear operators on it.

\subsection{Bidual of $\rC^*$-algebras and polar decomposition}\label{SS:C*alg}
Let $X$ be a normed space. We denote its dual space by $X^*$. If $\tau:X\to B(\fH)$ is a bounded linear map, then we denote by $\widehat\tau:X^{**}\to B(\fH)$ its unique weak-$*$ continuous linear extension.

Let $\fT$ be a $\rC^*$-algebra. Then, the bidual $\fT^{**}$ can be given the structure of a von Neumann algebra (see \cite[Theorem A.5.6]{BLM2004}). If $\A\subset \fT$ is a subalgebra of $\fT$, then $\A^{**}$ is a weak-$*$ closed subalgebra of $\fT^{**}$ and $\A$ is a subalgebra of $\A^{**}$. If $\pi:\fT\to B(\fH)$ is a $*$-representation, then $\widehat\pi:\fT^{**}\to B(\fH)$ is a weak-$*$ continuous $*$-representation.

Let $\tau:\fT\to \bC$ be a bounded linear functional. Then, $\widehat\tau:\fT^{**}\to \bC$ is a weak-$*$ continuous linear functional. In particular, it admits a polar decomposition. More precisely, there is a positive weak-$*$ continuous functional $\omega_\tau$ on $\fT^{**}$ along with a partial isometry $f\in \fT^{**}$ such that
\[
\widehat\tau(w)=\omega_{\tau}(f^* w), \quad w\in \fT^{**}.
\]
We remark that this is a slight departure from the convention used for instance in \cite[Theorem 1.14.4]{sakai1971} or \cite[Theorem III.4.2]{takesaki2002}, but is in line with that found in \cite[Definition 12.2.7]{dixmier1977}. At any rate, it is elementary to see that one convention is related to the other by simply considering the adjoint functional.

The restriction of $\omega_\tau$ to $\fT$ will be referred to as the \emph{absolute value} of $\tau$ and be denoted by $|\tau|$. We have that 
\[
\|\tau\|=\omega_\tau(f^*f)=\omega_\tau(I)=\widehat\tau(f).
\]
For ease of reference, we record the following standard fact.

\begin{lemma}
	Let $\fT$ be a unital $\Cstar$-algebra and let $\tau\in\fT^*$.
	Then, there exist a unital $*$-representation $\pi:\fT\to B(\fH)$, a vector $\xi\in\fH$ with $\|\xi\|=\|\tau\|^{1/2}$ that is cyclic for $\pi$ and a partial isometry $f\in\fT^{**}$ such that  $\wh{\pi}(f^*f)\xi=\xi$, 
	\[ \widehat\tau(w)=\ip{\widehat\pi(w)\xi,\wh{\pi}(f)\xi}, \qand  \omega_\tau (w)=\ip{\widehat\pi(w)\xi,\xi}, \quad w\in\fT^{**}\]
	for every $w\in \fT^{**}$.
	\label{L:PolarForm}
\end{lemma}
\begin{proof}
Applying the GNS construction to $|\tau|$, we obtain a $*$-representation $\pi:\fT\to B(\fH)$ along with a vector $\xi\in\fH$ with 
\[
\|\xi\|^2=|\tau|(I)=\omega_\tau(I)=\|\tau\|
\]
which is cyclic for $\pi$ and such that
\[
|\tau|(t)=\ip{\pi(t)\xi,\xi}, \quad t\in \fT.
\]
From this, it readily follows that
\[
\omega_{\tau}(w)=\ip{\widehat\pi(w)\xi,\xi}, \quad w\in \fT^{**}.
\]
whence
\[
\widehat\tau(w)=\omega_\tau(f^*w)=\ip{\widehat\pi(w)\xi,\widehat\pi(f)\xi}, \quad w\in \fT^{**}.
\]
Finally, because
\[
\langle \widehat\pi(f^*f)\xi,\xi \rangle=\omega_\tau(f^*f)=\|\tau\|=\|\xi\|^2
\]
the Cauchy--Schwarz inequality forces $\widehat\pi(f^*f)\xi=\xi$.
\end{proof}

We also require the following elementary uniqueness statement.

\begin{lemma}
	Let $\fT$ be a $\Cstar$-algebra and let $\tau:\fT\to\bC$ be a bounded linear functional. Assume that for $j=1,2$ we have a $*$-representation $\pi_j:\fT\to B(\fH_j)$ and vectors $\xi_j,\eta_j\in \fH_j$ with $\|\xi_j\|=\|\eta_j\|=\|\tau\|^{1/2}$ such that $\xi_j$ is cyclic for $\pi_j$ and	
		\[ \ip{\pi_1(t)\xi_1,\eta_1}=\tau(t)=\ip{\pi_2(t)\xi_2,\eta_2}, \quad t\in\fT. \]
	Then, there is a unitary operator $U:\H_1\to\H_2$ such that $U\xi_1=\xi_2, U\eta_1=\eta_2$ and 
	\[
	U\pi_1(t)U^*=\pi_2(t), \quad t\in \fT.
	\]
	\label{L:Subrep}
\end{lemma}
\begin{proof}
First note that by weak-$*$ compactness of the unital ball of $\fT^{**}$, there is a contraction $f_j\in \fT^{**}$ such that $\widehat\tau(f_j)=\|\tau\|$. By assumption, we find
\[
\langle\widehat{\pi_1}(f_1)\xi_1,\eta_1 \rangle=\|\tau\|=\langle\widehat{\pi_2}(f_2)\xi_2,\eta_2 \rangle.
\]
By the Cauchy--Schwarz inequality we then find 
\[
\widehat{\pi_j}(f_j)^*\eta_j=\xi_j \qand \widehat{\pi_j}(f_j)\xi_j=\eta_j.
\]
Therefore, for $t\in \fT$ we obtain
\begin{align*}
\langle \pi_1(t)\xi_1,\xi_1\rangle &=\langle \widehat{\pi_1}(f_1 t)\xi_1,\eta_1\rangle=\langle \widehat{\pi_2}(f_2 t)\xi_2,\eta_2\rangle\\
&=\langle \pi_2(t)\xi_2,\xi_2\rangle.
\end{align*}
The conclusion now follows immediately from \cite[Proposition 2.4.1]{dixmier1977}.
\end{proof}

\subsection{Structure of the non-commutative disc algebra}\label{SS:ncda}
We briefly recall the main features of the structure of the non-commutative disc algebra. For more details, the reader can consult \cite{popescu1991}, \cite{davidson2005} or \cite{kennedy2013} and the references therein.

Let $d$ be a positive integer and let $\bF^+_d$ denote the free monoid with generators $\{1,2,\ldots,d\}$.
Given a $d$-tuple $Z=(Z_1,\ldots,Z_d)$ of bounded linear operators on a common Hilbert space and 
$ w=j_1\cdots j_k\in \bF^+_d, $
we set
\[ Z_w=Z_{j_1}\cdots Z_{j_k}. \]
We also put $Z_\emptyset=I$.

The Hilbert space $\ell^2(\bF^+_d)$ is called the \textit{Fock space} and is denoted by $\fF^2_d$. For a (possibly empty) word $w\in \bF^+_d$, we let $e_w\in \fF^2_d$ denote the characteristic function of $\{w\}$. Then, $\{e_w\}_{w\in\bF^+_d}$ is an orthonormal basis for $\fF^2_d$. The left regular representation of $\bF^+_d$ gives rise to a $d$-tuple  $L=(L_1,\ldots,L_d)$ of operators on $\fF^2_d$ such that
\[ L_j e_w=e_{j w}, \quad w\in\bF^+_d, \, j=1,\ldots,d. \]
It is readily verified that $L_1,\ldots,L_d$ are isometries with pairwise orthogonal ranges.

The unital norm closed  algebra $\A_d\subset B(\fF^2_d)$ generated by $L_1,\ldots,L_d$ is Popescu's  \textit{noncommutative disc algebra}. 
The norm closed operator system that it generates inside of $B(\fF^2_d)$ is denoted by $\S_d$. Furthermore, we let $\fT_d=\Cstar(\A_d)\subset B(\fF^2_d)$; this is the Cuntz--Toeplitz algebra. It is a consequence of \cite[Theorem 1.3]{popescu2006} that $\fT_d$ contains the ideal $\K$ of compact operators on $\fF^2_d$. The corresponding quotient map implements a $*$-isomorphism between $\fT_d/\K$ and the Cuntz algebra $\O_d$. Thus, up to this identification, we obtain a quotient map $q:\fT_d\to \O_d$. Notably, $\S_d$ can be identified with its image under this map.

\begin{lemma}\label{L:OSnorm}
The quotient map $q:\fT_d\to \O_d$ restricts to a unital completely isometric linear isomorphism between $\S_d$ and $q(\S_d)$. 
\end{lemma}
\begin{proof}
It follows from \cite[Theorem 3.1]{popescu1996} that $q$ restricts to a unital completely isometric isomorphism between $\A_d$ and $q(\A_d)$. This then induces the desired map between operator systems by virtue of \cite[Proposition 3.5]{paulsen2002}.
\end{proof}

In light of this lemma, we may view $\S_d$ as embedded inside of $\O_d$. Such an identification will be used throughout the paper, and is of particular importance in Section \ref{S:FMRCNS}.

We let $\L_d\subset B(\fF^2_d)$ denote the closure of $\A_d$ in the weak operator topology. This is the prototypical example of a free semigroup algebra. Such algebras have been studied extensively since the appearance of \cite{davidson1997}, and much is known about their structure. By \cite[Theorem 2.6]{davidson2001}, there is a projection $\fq\in \A_d^{**}$ with the following properties.
\begin{enumerate}[{\rm (a)}]
\item We have $\fq a \fq=\fq a$ for every $a\in \A_d$.

\item There is a unital, completely isometric and weak-$*$ homeomorphic algebra isomorphism $\Phi:\fA_d^{**}(I-\fq)\to \L_d$ such that 
\[
\Phi(a(I-\fq))=a, \quad a\in \A_d.
\]
\end{enumerate}
Henceforth, we refer to $\fq$ as the \textit{universal structure projection}. 

We let $\V_d\subset B(\fF^2_d)$ denote the closure of $\S_d$ in the weak-$*$ topology. 
When $d\geq 2$, it is known that  $\V_d$  is very rigid.

\begin{theorem}\label{T:rigidKennedy}
Let $d\geq 2$ and let $\pi:\fT_d\to B(\fH)$ be a unital $*$-representation. Let $\L_\pi\subset B(\fH)$ denote the closure of $\pi(\A_d)$ in the weak operator topology, and let $\V_\pi\subset B(\fH)$ denote the weak-$*$ closed operator system generated by $\pi(\A_d)$. Assume that $\L_\pi$ is algebraically isomorphic to a subalgebra of $\L_d$. Then, there is a unital, completely isometric, weak-$*$ homeomorphic isomorphism $\Psi:\V_d\to \V_\pi$ extending $\pi$.
\end{theorem} 
\begin{proof}
This follows from \cite[Corollary 2.8]{davidson2001} and \cite[Theorem 3.6]{kennedy2011}.
\end{proof}

The following consequence will be very useful.

\begin{lemma}\label{L:extVd}
Let $d\geq 2$ and let $\tau:\S_d\to \bC$ be a bounded linear functional. Assume that
\[
\tau(a+b^*)=\widehat\tau(a(I-\fq)+(I-\fq)b^*), \quad a,b\in \A_d.
\]
Then, $\tau$ can be extended weak-$*$ continuously to $\V_d$.
\end{lemma}
\begin{proof}
By property (b) of the universal structure projection, there is a unital, completely isometric and weak-$*$ homeomorphic isomorphism $\Phi:\fA_d^{**}(I-\fq)\to \L_d$ such that 
\[
\Phi(a(I-\fq))=a, \quad a\in \A_d.
\]
Let $\R\subset\S_d^{**}$ be the weak-$*$ closure of
	\[ \{a(I-\fq)+(I-\fq)b^*:a,b\in\A_d\}. \]
By Theorem \ref{T:rigidKennedy}, there is a unital, completely isometric and weak$^*$-homeomorphic linear isomorphism $\Psi:\R\to\V_d$ such that
	\[ \Psi( a(I-\fq)+(I-\fq)b^* )=a+b^*, \quad a,b\in\A_d. \]
	Thus, we find
	\[ \tau(a+b^*)=\widehat\tau(\Psi^{-1}(a+b^*)) \quad a,b\in\A_d, \]
	whence the desired weak-$*$ continuous extension of $\tau$ to $\V_d$ is $\widehat\tau\circ \Psi^{-1}$.
\end{proof}


\subsection{Classification of $*$-representations of the Cuntz--Toeplitz algebra}\label{SS:classrep}

The representation theory of the Cuntz--Toeplitz algebra can be entirely understood in terms of \emph{row isometries}, that is, $d$-tuples $(V_1,\ldots,V_d)$ of isometries with pairwise orthogonal ranges. Indeed, let $\pi:\fT_d\to B(\fH)$ be a unital $*$-representation. Then, it is easily verified that $(\pi(L_1),\ldots,\pi(L_d))$ is a row isometry on $\fH$. Conversely, given a row isometry $V=(V_1,\ldots,V_d)$ on some Hilbert space $\fH$,  it follows from \cite[Theorem 1.3]{popescu1989} that there is a unital $*$-representation $\pi:\fT_d\to B(\fH)$ such that $\pi(L_j)=V_j$ for every $1\leq j\leq d$.  For our purposes in this paper, it is more convenient to use the language of $*$-representations of the $\rC^*$-algebras $\fT_d$ and $\O_d$. However, as explained above, this point of view is completely equivalent to one where the focus is placed on row isometries instead.

One consequence of the previous paragraph is as follows. Let $\pi:\fT_d\to B(\fH)$ be a unital $*$-representation and let $\fM\subset \fH$ be a closed subspace which is invariant for $\pi(\A_d)$. Then, the $d$-tuple $(\pi(L_1)|_\fM,\ldots,\pi(L_d)|_\fM)$ is a row isometry, so there is a unital $*$-representation $\rho:\fT_d\to B(\fM)$ such that $\rho(L_j)=\pi(L_j)|_\fM$ for every $1\leq j\leq d$. We will refer to $\rho$ as the representation \emph{induced by $\pi$ on $\fM$}.

If $\pi:\fT_d\to B(\fH)$ is a $*$-representation, then the argument given in the proof of \cite[Theorem 3.1]{popescu1996}  shows that the restriction of $\pi$ to $\A_d$ is completely isometric when $d\geq 2$. This very strong rigidity property of $\A_d$ implies in particular that different $*$-representations of $\fT_d$ cannot be distinguished by their behaviour on $\A_d$, so finer properties must be analyzed for this purpose.

We say that the representation $\pi$ is \emph{absolutely continuous} if there is a weak-$*$ continuous homomorphism $\pi':\fL_d\to B(\fH)$ such that
\[
\pi'(a)=\pi(a), \quad a\in \A_d.
\]
%
%
%
%
We say that the representation $\pi$ is \emph{singular} if there is no non-zero closed subspace of $\fH$ which is invariant for $\pi(\A_d)$ and such that the corresponding representation induced by $\pi$ is absolutely continuous. 
We will require the following known fact, relating absolute continuity and singularity to the universal structure projection $\fq$.

\begin{lemma}\label{L:UnivStructProj}
	Let $\pi:\fT_d\to B(\fH)$ be a unital $*$-representation. Let $\fM\subset\fH$ be a closed invariant subspace for $\pi(\A_d)$ and let $\rho:\fT_d\to B(\fM)$ be the corresponding representation induced by $\pi$. Let $P\in B(\fH)$ denote the orthogonal projection onto $\fM$. Then, the following statements hold.
\begin{enumerate}[{\rm (i)}]
\item The representation $\rho$ is absolutely continuous if and only if $\widehat{\pi}(\fq)P=0$.

\item The representation $\rho$ is singular if and only if $\widehat{\pi}(\fq)P=P$.

\end{enumerate}
\end{lemma}
\begin{proof}
Recall first that by property (b) in Subsection \ref{SS:ncda}, there is  unital, completely isometric and weak-$*$ homeomorphic isomorphism $\Phi:\fA_d^{**}(I-\fq)\to \L_d$ such that 
\[
\Phi(a(I-\fq))=a, \quad a\in \A_d.
\]

(i) Assume that $\rho$ is absolutely continuous. By Goldstine's theorem, there is a contractive net $(a_i)$ in $\A_d$ converging to $\fq$ in the weak-$*$ topology of $\A_d^{**}$. Note then that $(a_i)$ converges to $0$ in the weak-$*$ topology of $\L_d$ by the previous paragraph, and thus $(\rho(a_i))$ converges to $0$ in the weak-$*$ topology of $B(\fM)$.  For $\xi\in \fM$ we thus find
\begin{align*}
\|\widehat\pi(\fq)\xi\|^2&=\langle \widehat\pi(\fq)\xi,\xi\rangle=\lim_i \langle \pi(a_i)\xi,\xi\rangle\\
&=\lim_i \langle \rho(a_i)\xi,\xi\rangle=0
\end{align*}
so that $\widehat\pi(\fq)P=0$.

Conversely, assume that $\widehat{\pi}(\fq)P=0$. For $a\in \A_d$, we find
\begin{align*}
\rho(a)=\pi(a)P|_\fM=\widehat\pi(a(I-\fq))P|_\fM=(\widehat\pi\circ \Phi^{-1})(a)P|_\fM
\end{align*}
so that $\rho$ is absolutely continuous.

(ii) Assume that $\rho$ is singular. Since $\widehat\pi(\fq)$ lies in the weak-$*$ closure of $\pi(\A_d)$, we infer that $\widehat\pi(\fq)\fM\subset \fM$. By property (a) in Subsection \ref{SS:ncda}, it follows that the closed subspace $\fN= \ol{\widehat\pi(I-\fq)\fM}$ is invariant for $\pi(\A_d)$. 
We see that
\[
\rho(a)|_{\fN}=\widehat\pi (a(I-\fq))|_\fN=(\widehat\pi\circ \Phi^{-1})(a)|_\fN
\]
for every $a\in \A_d$. We infer that the representation induced by $\rho$ on $\fN$ is absolutely continuous, whence $\fN=\{0\}$ since $\rho$ is assumed to be singular. Thus, $\widehat\pi(I-\fq)P=0$ and so $P=\widehat\pi(\fq) P$.

Conversely, assume that $P=\widehat\pi(\fq) P$. Let $\fN\subset \fM$ be a closed invariant subspace for $\rho(\A_d)$ such that the corresponding representation induced by $\rho$ is absolutely continuous. Let $Q\in B(\fH)$ denote the orthogonal projection onto $\fN$. By (i), we infer that $\widehat\pi(\fq)Q=0$. Thus
\begin{align*}
Q=QP=Q\widehat{\pi}(\fq)P=0.
\end{align*}
We conclude that $\rho$ is singular.
\end{proof}

%
%
%

The concepts of absolute continuity and of singularity for $*$-representations are motivated by the corresponding notions from measure theory. But this analogy has some limitations: when $d\geq 2$, there is a third type of $*$-representations of $\fT_d$. A unital $*$-representation $\pi:\fT_d\to B(\fH)$ is said to be of \emph{dilation type} if there is no non-zero closed subspace of $\fH$ which is reducing for $\pi(\fT_d)$ and such that the corresponding representation induced by $\pi$ is absolutely continuous or singular. 

For our purposes, we also require the following terminology. The representation $\pi$ is said to be \emph{completely non-singular} if there is no non-zero closed subspace of $\fH$ which is reducing for $\pi(\fT_d)$ and such that the corresponding representation induced by $\pi$ is singular.

If $\sigma:\O_d\to B(\fH)$ is a unital $*$-representation and  $q:\fT_d\to \O_d$ denotes the quotient map, then we say that $\sigma$ is absolutely continuous, singular, completely non-singular or of dilation type whenever $\sigma\circ q$ has the corresponding property.

 Recall that a vector $\xi\in\fH$ is a \textit{wandering vector} for some unital $*$-representation $\pi:\fT_d\to B(\fH)$ if $\{\pi(L_w)\xi:w\in\bF^+_d\}$ is an orthonormal set in $\fH$. It is readily seen that if $\xi$ is wandering for $\pi$, then the representation induced by $\pi$ on the invariant subspace $\ol{\pi(\A_d)\xi}$ is unitarily equivalent to the identity representation of $\fT_d$. 
We now consider an important family of examples, which is studied in detail in \cite[Section 3]{davidson1999}.

\begin{example}\label{E:AtomicD}
Let $m\in \bN$. Fix a map 
\[
\upsilon:\{1,\ldots,m\}\to\{1,2,\ldots,d\}
\]
and  a complex number $\lambda$ with $|\lambda|=1$.  Let $\fH$ be a Hilbert space with basis
	\[ \{\xi_{n,w}:1\leq n\leq m,w\in\bF^+_d\bksl \bF^+_d\upsilon(n)\}. \]
	For $1\leq j\leq d$, we define an isometry $S_j\in B(\fH)$ by putting
		\[
		S_j\xi_{n,w}=\begin{cases}
			\lambda\xi_{1,\emptyset} & \text{if }  n=m,w=\emptyset \text{ and }j=\upsilon(m), \\
			\xi_{n+1,\emptyset} & \text{if }  n<m,w=\emptyset \text{ and }j=\upsilon(n), \\
			\xi_{n,j} &  \text{if } w=\emptyset \text{ and }j\neq \upsilon(n), \\
			\xi_{n,jw} &  \text{if } w\neq \emptyset.
		\end{cases}
	\]
	By construction, the isometries $S_1,\ldots,S_d$ have pairwise orthogonal ranges that span $\fH$. Thus, the $\rC^*$-algebra that they generate is $*$-isomorphic to $\O_d$, so there is a unital $*$-representation $\sigma_{\upsilon,\lambda}:\O_d\to B(\fH)$ with $\sigma_{\upsilon,\lambda}(q(L_j))=S_j$ for $1\leq j\leq d$. 
	
Consider now the word 
\[
u=\upsilon(m)\upsilon(m-1)\cdots \upsilon(1)\in \bF^+_d.
\]
If $\upsilon$ is chosen so that there is no word $w\in \bF^+_d$  and no integer $p>1$ such that $u=w^p$ (in other words, $u$ is primitive), then the representation $\sigma_{\upsilon,\lambda}$ is known to be irreducible \cite[Proposition 3.10]{davidson1999}. We claim then that $\sigma_{\upsilon,\lambda}$ is of dilation type; by definition, this means that we must verify that $\sigma_{\upsilon,\lambda}$ is neither absolutely continuous nor singular.

It is easily seen that $\xi_{n,w}$ is a wandering vector for $\sigma_{\upsilon,\lambda}$ whenever $w\neq \emptyset$. In particular, the representation induced by $\sigma_{\upsilon,\lambda}$ on the corresponding cyclic invariant subspace is unitarily equivalent to the identity representation of $\fT_d$ and hence absolutely continuous, so $\sigma_{\upsilon,\lambda}$ is not singular. If $\sigma_{\upsilon,\lambda}$ were absolutely continuous, then the sequence $(\sigma_{\upsilon,\lambda}(L_u^n))_n$ would converge to $0$ in the weak-$*$ topology of $B(\fH)$. But this is not the case, seeing as
	\[
	\sigma_{\upsilon,\lambda}(L_u^n)\xi_{1,\emptyset}=\lambda^n \xi_{1,\emptyset}, \quad n\in \bN.
	\]
We thus conclude that $\sigma_{\upsilon,\lambda}$ is of dilation type, as claimed. Note also that different values of $\lambda$ produce unitarily inequivalent $*$-representations; see the proof of \cite[Theorem 3.4]{davidson1999}. 
\qed
\end{example}

\section{The classical F. and M. Riesz theorem}\label{S:classFM}
We begin  by stating the classical theorem that we seek to generalize (see \cite[page 47]{hoffman1969}). 

\begin{theorem}[F. and M. Riesz]\label{T:classFM}
Let $\mu$ be a regular Borel measure on the unit circle $\bT$. Assume that
\[
\int_{\bT}ad\mu=0, \quad a\in \rA(\bD).
\]
Then, $\mu$ is absolutely continuous with respect to Lebesgue measure $\lambda$.
\end{theorem}

We now aim to reframe this theorem in operator algebraic terms that are more suitable for our generalizations to come. In addition, this reframing will provide motivation for our main results.

\subsection{Completely non-singular representations}
The $\rC^*$-algebra $\rC(\bT)$ acts on $L^2(\bT,\lambda)$ as multiplication operators. This yields a unital injective $*$-representation $\beta:\O_1\to B(L^2(\bT,\lambda))$. It is easy to see that $\beta$ is an absolutely continuous representation in the language of Subsection \ref{SS:classrep}. In the univariate setting, this is equivalent to $\beta$ being completely non-singular. We now record a reformulation of Theorem \ref{T:classFM}.

\begin{corollary}\label{C:classFM1}
Let $\tau:\O_1\to\bC$ be a bounded linear functional with $q(\A_1)\subset \ker \tau$. Then, there are unit vectors $\xi,\eta\in L^2(\bT,\lambda)$ with $\|\xi\|=\|\eta\|=\|\tau\|^{1/2}$ such that $\xi$ is cyclic for $\beta$ and
\[
\tau(t)=\langle \beta(t)\xi,\eta\rangle, \quad t\in \O_1.
\]
Furthermore, we have the following uniqueness property: if $\pi:\O_1\to B(\fH')$ is a unital $*$-representation and $\xi',\eta'\in \fH'$ are vectors with $\|\xi'\|=\|\eta'\|=\|\tau\|^{1/2}$ such that $\xi'$ is cyclic for $\pi$ and
\[
\tau(t)=\langle \pi'(t)\xi',\eta'\rangle, \quad t\in \O_1,
\]
then there is a unitary operator $U:L^2(\bT,\lambda)\to\fH'$ such that $U\xi=\xi', U\eta=\eta'$ and
\[
U\beta(t)U^*=\pi'(t), \quad t\in \O_1.
\]

\end{corollary}
\begin{proof}
Under the natural identification $\O_1\cong\rC(\bT)$, we see that $\rA(\bD)$ is identified with $q(\A_1)$. Thus, applying Theorem \ref{T:classFM} we find a function $r\in L^1(\bT,\lambda)$ such that $\|r\|_1=\|\tau\|$ and
\[
\tau(t)=\int_\bT t(\zeta)r(\zeta)d\lambda(\zeta), \quad t\in \O_1.
\]
By assumption, we find
\[
\int_{\bT}\zeta^n r(\zeta)d\lambda(\zeta)=0, \quad n\geq 0
\]
which implies that $r\in H^1$ \cite[page 39]{hoffman1969}. By \cite[Corollary 6.27]{douglas1998}, there are functions $\xi, \eta\in L^2(\bT,\lambda)$ which are non-zero $[\lambda]$-a.e on $\bT$, such that $\|\xi\|_2=\|\eta\|_2=\|\tau\|^{1/2}$ and such that $r=\xi\ol{\eta}$. In particular, this implies that
\[
\tau(t)=\langle \beta(t)\xi,\eta\rangle, \quad t\in \O_1.
\]
Furthermore, it follows from \cite[Theorem 6.8]{douglas1998} that $\xi$ does not belong to any non-zero reducing subspace for $\beta(\O_1)$, whence $\xi$ is cyclic for $\beta$.

The uniqueness statement follows at once from Lemma \ref{L:Subrep}.

\end{proof}

Generalizing this corollary to the multivariate setting where $d>1$ will be the focus of Section \ref{S:FMRCNS}.

\subsection{Weak-$*$ continuous extensions}

The von Neumann algebra $\W_\beta$ generated by $\beta(\O_1)$ inside of $B(L^2(\bT,\lambda))$ is $*$-isometric and weak-$*$ homeomorphic to $L^\infty(\bT,\lambda)$. Theorem \ref{T:classFM} is thus equivalent to the fact that the functional $\tau$ on $\O_1\cong \rC(\bT)$ given as integration against $\mu$ has the property that $\tau\circ \beta^{-1}$ extends weak-$*$ continuously to $\W_\beta$. In this subsection, our aim is to make such a statement more canonical by removing the dependence on the representation $\beta$. At least at the level of operator systems, this is indeed possible because of a strong rigidity property, very much in line with Theorem \ref{T:rigidKennedy}. Unfortunately, the machinery behind the proof of Theorem \ref{T:rigidKennedy} requires that $d\geq 2$. In what follows, we show that a similar rigidity result still holds when $d=1$, using classical tools. We suspect that the statement itself may be known to experts, but we lack an appropriate reference so we provide the details.

When $d=1$, it is well known that $\A_1$ can be identified, up to unitary equivalence, as an algebra of multiplication operators on the space $H^2\subset L^2(\bT,\lambda)$. Here, $H^2$ is the standard Hardy space, defined as the closure in the $L^2$-norm of the space of holomorphic polynomials. For $f\in L^\infty(\bT,\lambda)$, we also recall that the corresponding Toeplitz operator $T_f\in B(H^2)$ is defined as
\[
T_f h=P_{H^2} fh, \quad h\in H^2.
\]
We start with an elementary observation. Recall that $\V_1\subset B(H^2)$ is the weak-$*$ closed operator system generated by $\A_1$.
\begin{lemma}\label{L:Toeplitz}
Define $\Xi:L^\infty(\bT,\lambda)\to B(H^2)$ as
\[
\Xi(f)=T_f, \quad f\in L^\infty(\bT,\lambda).
\]
Then, $\Xi$ is a unital completely isometric linear map which is a weak-$*$ homeomorphic surjection onto $\V_1$.
\end{lemma}
\begin{proof}
The fact that $\Xi$ is linear and isometric can be found in \cite[Corollary 7.8]{douglas1998}. 
It is also readily seen that $\Xi$ is completely contractive. Because $L^\infty(\bT,\lambda)$ is a commutative $\rC^*$-algebra, we may invoke \cite[Theorem 3.9]{paulsen2002} to see that $\Xi$ is in fact completely isometric.

 Let $(f_\alpha)$ be a bounded net in $L^\infty(\bT,\lambda)$ converging to some $f$ in the weak-$*$ topology. Let $g,h\in H^2$. Then, $g\ol{h}\in L^1(\bT,\lambda)$, so that
\begin{align*}
\lim_\alpha \langle T_{f_\alpha}g,h\rangle&=\lim_\alpha \int_\bT f_\alpha g\ol{h}d\lambda=\int_\bT fg\ol{h}d\lambda\\
&=\langle T_fg,h \rangle.
\end{align*}
Because $\Xi$ is isometric, the net $(T_{f_\alpha})$ is bounded and we infer that it converges to $T_f$ in the weak-$*$ topology of $B(H^2)$. Basic functional analytic arguments reveal that $\Xi$ is a weak-$*$ homeomorphism onto its range \cite[Theorem A.2.5]{BLM2004}. Finally, we note that $L^\infty(\bT,\lambda)$ is the weak-$*$ closed operator system generated by $\rA(\bD)$, whence the range of $\Xi$ is the weak-$*$ closed operator system generated by $\A_1$, which is $\V_1$.
\end{proof}
%
%
%
%
Next, we establish a useful fact about absolutely continuous representations of $\O_1$.

\begin{lemma}\label{L:genunitary}
Let $\fH$ be separable Hilbert space and let $\pi:\O_1\to B(\fH)$ be an absolutely continuous unital $*$-representation. 
Then, there is a unital, weak-$*$ continuous, completely contractive linear map $\Psi:\V_1\to B(\fH)$ such that
\[
\Psi(a)=\pi(q(a)), \quad a\in \A_1.
\]
\end{lemma}
\begin{proof}
Let $\W_\pi\subset B(\fH)$ denote the von Neumann algebra generated by $\pi(\O_1)$.
By \cite[Theorem II.2.5]{davidson1996}, there is a regular Borel probability measure $\nu$ on $\bT$ and a weak-$*$ homeomorphic $*$-isomorphism $\Theta:L^\infty(\bT,\nu)\to \W_\pi$ such that
\[
\Theta (t)=\pi(t), \quad t\in \O_1.
\]
Because $\pi$ is assumed to be absolutely continuous, the map
\[
a\mapsto \Theta^{-1}(\pi(q(a))), \quad a\in \A_1
\]
extends weak-$*$ continuously to $\L_1$, which implies that $\nu$ is absolutely continuous with respect to $\lambda$; this is a consequence of Theorem \ref{T:classFM} along with \cite[Theorem 9.2.1]{rudin2008}.  Thus, there is a weak-$*$ continuous unital $*$-homomorphism $\rho:L^\infty(\bT,\lambda)\to L^\infty(\bT,\nu)$ with the property that for each $f\in L^\infty(\bT,\lambda)$ we have
\[
\rho(f)(\zeta)=f(\zeta)
\]
for $[\nu]$-a.e. $\zeta\in \bT$. 
Recall now from Lemma \ref{L:Toeplitz} that there is a unital, weak-$*$ homeomorphic and completely isometric linear map $\Xi:L^\infty(\bT,\lambda)\to \V_1$ such that
\[
\Xi(q(a))=a, \quad a\in \A_1.
\]
The map
$
\Psi= \Theta\circ\rho \circ \Xi^{-1}:\V_1\to \W_\pi
$
has all the required properties.
\end{proof}

Before proving our announced rigidity result, we require one more technical fact.

\begin{lemma}\label{L:unitary}
Let $\fH$ be separable Hilbert space and let $\pi:\O_1\to B(\fH)$ be a unital $*$-representation. Let $\L_\pi\subset B(\fH)$ denote the closure of $\pi(q(\A_1))$ in the weak operator topology, let $\V_\pi\subset B(\fH)$ denote the weak-$*$ closed operator system generated by $\pi(q(\A_1))$, and let $\W_\pi\subset B(\fH)$ denote the von Neumann algebra generated by $\pi(q(\A_1))$. Then, the following statements are equivalent.
\begin{enumerate}[{\rm (i)}]
\item The algebra $\L_\pi$ is algebraically isomorphic to a subalgebra of $\L_1$.
\item There is a unital, weak-$*$ homeomorphic, completely isometric linear isomorphism $\Psi:\V_1\to \V_\pi$ such that 
\[
\Psi(a)=\pi(q(a)), \quad a\in \A_1.
\]
\item There is a regular Borel probability measure $\nu$ on $\bT$ that is mutually absolutely continuous with respect to Lebesgue measure and a weak-$*$ homeomorphic $*$-isomorphism $\Theta:L^\infty(\bT,\nu)\to \W_\pi$ such that
\[
\Theta (t)=\pi(t), \quad t\in \O_1.
\]
\end{enumerate}
\end{lemma}
\begin{proof}
By Lemma \ref{L:Toeplitz} we find a unital, weak-$*$ homeomorphic and completely isometric linear map $\Xi:L^\infty(\bT,\lambda)\to \V_1$ such that 
\[
\Xi(q(a))=a, \quad a\in \A_1.
\]
 This notation will be used throughout the proof. 

%

(iii) $\Rightarrow$ (ii): By assumption, there is a weak-$*$ homeomorphic $*$-isomorphism $\rho: L^\infty(\bT,\lambda)\to L^\infty(\bT,\nu)$ such that 
\[
\rho(t)=t, \quad t\in \O_1.
\]
In particular, 
\[
\V_\pi=\Theta\circ \rho \circ \Xi^{-1}(\V_1)=\W_\pi.
\] 
Thus, we may take $\Psi=\Theta\circ \rho\circ \Xi^{-1}$.

(ii) $\Rightarrow$ (i): This is trivial.

(i) $\Rightarrow$ (iii):  By \cite[Corollary 2.8]{davidson2001}, we conclude that there  is a unital, weak-$*$ homeomorphic, completely isometric isomorphism $\Phi:\L_1\to \L_\pi$ such that 
\[
\Phi(a)=\pi(q(a)), \quad a\in \A_1.
\]
In turn, by \cite[Theorem II.2.5]{davidson1996}, there is a regular Borel probability measure $\nu$ on $\bT$ and a weak-$*$ homeomorphic $*$-isomorphism $\Theta:L^\infty(\bT,\nu)\to \W_\pi$ such that
\[
\Theta (t)=\pi(t), \quad t\in \O_1.
\]
Arguing as in the proof of Lemma \ref{L:genunitary} above, we infer that $\nu$ is absolutely continuous with respect to $\lambda$ and that  there is a  weak-$*$ continuous unital $*$-homomorphism $\rho:L^\infty(\bT,\lambda)\to L^\infty(\bT,\nu)$ such that for $f\in L^\infty(\bT,\lambda)$ we have
\[
\rho(f)(\zeta)=f(\zeta)
\]
for $[\nu]$-a.e. $\zeta\in \bT$. 

Let $E\subset \bT$ be a measurable subset such that $\nu(E)=0$. There is a  function $\gamma\in  \Xi^{-1}(\L_1)$ such that $|\gamma|=2\chi_E+\chi_{\bT\setminus E}$ $[\lambda]$-a.e. on $\bT$ \cite[page 62]{hoffman1969}. Choose a sequence $(a_n)$ in $\A_1$ such that $(\Xi^{-1}(a_n))$ converges to $\gamma$ in the weak-$*$ topology of $L^\infty(\bT,\lambda)$. Thus, 
\begin{align*}
(\Phi\circ\Xi)(\gamma)&=\lim_{n\to\infty}\Phi(a_n)=\lim_{n\to\infty}\pi(q(a_n))\\
&=\lim_{n\to\infty}(\Theta\circ \rho \circ \Xi^{-1})(a_n)\\
&=(\Theta\circ \rho)(\gamma).
\end{align*}
But we know that $\Phi$ is isometric, whence
\begin{align*}
\|\gamma\|_{L^\infty(\bT,\lambda)}&=\|\Xi(\gamma)\|=\|(\Phi\circ\Xi)(\gamma)\|\\
&=\|(\Theta\circ \rho)(\gamma)\|=\|\rho(\gamma)\|\\
&=\|2\chi_E+\chi_{\bT\setminus E}\|_{L^\infty(\bT,\nu)}=1
\end{align*}
where the last equality follows from the fact that $\nu(E)=0$. Thus, $\|\gamma\|_{L^\infty(\bT,\lambda)}=1$, which forces $\lambda(E)=0$. We conclude that $\lambda$ is absolutely continuous with respect to $\nu$, as desired.

\end{proof}

%

%
%

We can now prove our rigidity result, which is a univariate version of Theorem \ref{T:rigidKennedy}.

\begin{theorem}\label{T:rigidity}
Let $\fH$ be separable Hilbert space and let $\pi:\fT_1\to B(\fH)$ be a unital $*$-representation. Let $\L_\pi\subset B(\fH)$ denote the closure of $\pi(\A_1)$ in the weak operator topology, and let $\V_\pi\subset B(\fH)$ denote the weak-$*$ closed operator system generated by $\pi(\A_1)$.  Then, the following statements are equivalent.
\begin{enumerate}[{\rm (i)}]
\item The algebra $\L_\pi$ is algebraically isomorphic to a subalgebra of $\L_1$.
\item There is a unital, weak-$*$ homeomorphic, completely isometric linear isomorphism $\Psi:\V_1\to \V_\pi$ such that 
\[
\Psi(a)=\pi(a), \quad a\in \A_1.
\]
\end{enumerate}
\end{theorem}
\begin{proof}
We only need to show that (i) implies (ii), as the converse is trivial. Thus, assume that $\L_\pi$ is algebraically isomorphic to a subalgebra of $\L_1$. By \cite[Corollary 2.8]{davidson2001}, we conclude that there  is a unital, weak-$*$ homeomorphic, completely isometric isomorphism $\Phi:\L_1\to \L_\pi$ such that 
\[
\Phi(a)=\pi(a), \quad a\in \A_1.
\]
By the Wold decomposition, up to unitary equivalence we may write $\fH=H^{2(\kappa)}\oplus \fK$ and  $\pi=\id^{(\kappa)}\oplus (\sigma\circ q)$ for some cardinal $\kappa$ and some unital $*$-representation $\sigma:\O_1\to B(\fK)$. If $\kappa=0$, then the statement follows at once from Lemma \ref{L:unitary}. 

Assume henceforth that $\kappa>0$. The existence of $\Phi$ implies in particular that $\pi$ is absolutely continuous, and hence so is $\sigma\circ q$.
We may thus apply Lemma \ref{L:genunitary} to obtain  a unital, weak-$*$ continuous, completely contractive linear map $\Psi_0:\V_1\to B(\fK)$ such that 
\[
\Psi_0(a)=\sigma(q(a)), \quad a\in \A_1. 
\]
Consider now the map $\Psi:\V_1\to \V_1^{(\kappa)}\oplus B(\fK)$ defined as
\[
\Psi(v)=v^{(\kappa)}\oplus \Psi_0(v), \quad v\in \V_1.
\]
Clearly, $\Psi$ is unital, weak-$*$ continuous and completely isometric. By  \cite[Theorem A.2.5]{BLM2004}, we infer that $\Psi$ is weak-$*$ homeomorphic onto its image.  
Finally, we claim that $\Psi(\V_1)=\V_\pi$. To see this, note that 
\[
\Psi(a)=a^{(\kappa)}\oplus \sigma(q(a))=\pi(a) \quad a\in \A_1.
\]
Then, $\Psi(\V_1)$ is simply the weak-$*$ closure of the operator system generated by $\pi(\A_1)$, which is $\V_\pi$.
\end{proof}

As a consequence of this rigidity result, we can extract another consequence of Theorem \ref{T:classFM}, which will motivate some of our work below.

\begin{corollary}\label{C:classFM2}
Let $\tau:\S_1\to \bC$ be a bounded linear functional with the property that $\A_1\subset \ker \tau$. Then, $\tau$ extends weak-$*$ continuously to $\V_1$.
\end{corollary}
\begin{proof}
Consider the unital injective $*$-representation $\beta:\O_1\to B(L^2(\bT,\lambda))$ as before.
 Let $\L_\beta\subset B(L^2(\bT,\lambda))$ denote the closure of $\beta(q(\A_1))$ in the weak operator topology, let $\V_\beta\subset B(L^2(\bT,\lambda))$ denote the weak-$*$ closed operator system generated by $\beta(q(\A_1))$ and let $\W_\beta\subset B(L^2(\bT,\lambda))$ denote the von Neumann algebra generated by $\beta(q(\A_1))$.
 
It is well known that $\L_\beta$ is unitarily equivalent to $\L_1$. Theorem \ref{T:rigidity} then implies that there is a unital, weak-$*$ homeomorphic, completely isometric linear isomorphism $\Psi:\V_1\to \V_\beta$ such that 
\[
\Psi(a)=\beta(q(a)), \quad a\in \A_1.
\]
Let $\tau'$ be a Hahn--Banach extension to $\beta(\O_1)$ of the functional
\[
\beta(q(s))\mapsto (\tau\circ \Psi^{-1})(\beta(q(s))), \quad s\in \S_1.
\]
As noted at the beginning of this subsection, an immediate consequence of Theorem \ref{T:classFM} is that there is a weak-$*$ continuous functional $\tau'':\W_\beta\to \bC$ extending $\tau'$. The functional $\tau''\circ\Psi$ on $\V_1$ is then weak-$*$ continuous and satisfies
\[
(\tau''\circ \Psi)(s)=\tau''(\beta(q(s))=(\tau\circ \Psi^{-1})(\beta(q(s)))=\tau(s), \quad s\in \S_1
\]
as desired.
\end{proof}

%
%
%

Generalizing this corollary to the multivariate setting where $d>1$ will be the focus of Section \ref{S:FMRnAC}.

\section{Riesz representations for analytic functionals}\label{S:FMRCNS} 
The goal of this section is to prove a generalization of Corollary \ref{C:classFM1}. As explained in Section \ref{S:classFM}, this can be seen as a generalization of the classical F. and M. Riesz theorem.

\subsection{Completely non-singular representations}\label{SS:CNSrep}
In this short subsection, we establish some useful properties of completely non-singular representations that are of independent interest.

\begin{proposition}\label{P:CNSrestriction}
Let $\pi:\fT_d\to B(\fH)$ be a unital $*$-representation. Let $\xi\in \fH$ be a cyclic vector for $\pi$ and let $\rho:\fT_d\to B(\fH)$ be the representation induced by $\pi$ on the invariant subspace $\ol{\pi(\A_d)\xi}$. If $\rho$ is absolutely continuous, then $\pi$ is completely non-singular.
\end{proposition}
\begin{proof}
 Let $\fR\subset \fH$ be a closed reducing subspace for $\pi(\fT_d)$, and let $P\in B(\fH)$ be the orthogonal projection onto $\fR$. Assume that the representation induced by $\pi$ on $\fR$ is singular and let $Q\in B(\fH)$ denote the orthogonal projection onto $\ol{\pi(\A_d)\xi}$. It follows from Lemma \ref{L:UnivStructProj} that $P=P\widehat\pi(\fq)$ and $\widehat\pi(\fq)Q=0$. Therefore, $PQ=0$ so in particular $P\xi=0$. Since $P$ is reducing for $\pi(\fT_d)$, we infer that $P\pi(t)\xi=\pi(t)P\xi=0$ for every $t\in \fT_d$, and thus $P=0$ since $\xi$ is cyclic for $\pi(\fT_d)$.
\end{proof}

The previous result has a sort of converse when $d\geq 2$, stating that any completely non-singular representation is a direct sum of cyclic representations inducing absolutely continuous representations. In fact, we show something a bit more precise.

\begin{proposition}\label{P:CNSdirectsum}
Let $d\geq 2$. Let $\pi:\fT_d\to B(\fH)$ be completely non-singular unital $*$-repre\-sentation. Then, there is a set $\{\xi_\lambda:\lambda\in \Lambda\}$ of vectors in $\fH$ with the following properties.
\begin{enumerate}[{\rm (a)}]
\item For each $\lambda\in \Lambda$, we let $\fH_\lambda=\ol{\pi(\fT_d)\xi_\lambda}$ and let $\pi_\lambda$ denote the representation induced by $\pi$ on $\fH_\lambda$.
 Then, $\fH=\oplus_{\lambda\in \Lambda}\fH_\lambda$ and $\pi=\oplus_{\lambda\in \Lambda}\pi_\lambda$.

\item For each $\lambda\in \Lambda$, the representation induced by $\pi_\lambda$ on the invariant subspace $\ol{\pi_\lambda(\A_d)\xi_\lambda}$ is unitarily equivalent to the identity representation of $\fT_d$.
\end{enumerate}
\end{proposition}
\begin{proof}
An application of Zorn's lemma yields a maximal family $\F=\{\fH_\lambda:\lambda\in \Lambda\}$ of pairwise orthogonal closed reducing subspaces for $\pi(\fT_d)$ such that the representation induced by $\pi$ on each $\fH_\lambda$ admits a cyclic vector $\xi_\lambda$ which is wandering for $\pi$. Let $\fM=\left(\oplus_{\lambda\in \Lambda}\fH_\lambda \right)^\perp$. If $\fM$ is non-zero, then the representation induced by $\pi$ on $\fM$ is not singular by assumption, so that there is a wandering vector $\eta\in \fM$ for $\pi$ by virtue of \cite[Corollary 4.13]{kennedy2011} and \cite[Theorem 5.1]{kennedy2013} (this is where we require $d\geq 2$). By maximality, this forces $\ol{\pi(\fT_d)\eta}\in \F$, which is absurd. Hence $\fM$ is zero and we find $\pi=\oplus_{\lambda}\pi_\lambda$. Finally, since $\xi_\lambda$ is wandering for $\pi$, it immediately follows that the representation induced by $\pi_\lambda$ on $\ol{\pi_\lambda(\A_d)\xi_\lambda}$ is unitarily equivalent to the identity representation of $\fT_d$.
\end{proof}

We note that the previous proposition fails when $d=1$. Indeed, there are absolutely continuous $*$-representation of $\fT_1$ with no induced representations on invariant subspaces that are unitarily equivalent to the identity representation of $\fT_1$; see for instance \cite[Example 2.2]{kennedy2013}.

\subsection{Riesz representations}\label{SS:Rieszrep}
We now get back to extending the classical F. and M. Riesz theorem. The next lemma is our main technical tool.

\begin{lemma}
	Let $\pi:\fT_d\to B(\fH)$ be a unital $*$-representation. 
	Suppose there exist a vector $\xi\in\fH$ along with $f\in\fT_d^{**}$ such that $\wh{\pi}(f^*f)\xi=\xi$ and $\wh{\pi}(f)\xi\in \left(\pi(\A_d)\xi\right)^\perp$. If we let $\fM=\ol{\pi(\A_d)\xi}$, then the following statements hold.
	\begin{enumerate}[{\rm (i)}]
	\item The representation induced by $\pi$ on $\fM$ is unitarily equivalent to the identity representation of $\fT_d$.	
	
	\item We have $\widehat\pi(\fq)\xi=0$.
	
	\item The map
	$s\mapsto P_{\fM}\pi(s)|_{\fM}$ on $\S_d$ admits a weak-$*$ continuous extension to $\V_d$.
	\end{enumerate}
	\label{L:Main}
\end{lemma}
\begin{proof}

Let $\rho$ denote the representation induced  by $\pi$ on $\fM$. By virtue of the Wold decomposition \cite[Theorem 1.3]{popescu1989}, up to unitary equivalence we know that $\fM=\fF^{2(\kappa)}_d\oplus \fK$ and that $\rho=\id^{(\kappa)}\oplus (\sigma \circ q)$ for some cardinal $\kappa$ and some unital $*$-representation $\sigma:\O_d\to B(\fK)$. Let $P\in B(\fH)$ denote the orthogonal projection onto $\fK$. We claim that $P$ commutes with $\pi(\fT_d)$.	To see this, consider the unital completely positive map $\phi:\fT_d\to B(\fK)$ defined as
	\[
	\phi(t)=P_{\fK}\pi(t)|_{\fK}, \quad t\in \fT_d.
	\]
For $a\in \A_d$, we see that $\phi(a)=\sigma(q(a)).$ In particular, we see that  $\phi(L_1),\ldots,\phi(L_d)$ are isometries satisfying the Cuntz relation. It is then well known that  $\phi$ must be a $*$-homomorphism; indeed, it follows from the Schwarz inequality that $\fT_d$ lies in the multiplicative domain of $\phi$ \cite[Theorem 3.18]{paulsen2002} . In turn, a classical observation of Agler \cite{agler1982} (see also \cite[Lemma 3.2]{CH2018}) reveals that $\fK$ is reducing for $\pi(\fT_d)$, so that indeed $P$ commutes with $\pi(\fT_d)$. 
	It follows that
	\[ \wh{\pi}(f)^*P=P\wh{\pi}(f)^* \]
	and, for every $a\in\A_d$, that
	\[ P\pi(a)\xi=P\pi(a)\wh{\pi}(f^*f)\xi=\wh{\pi}(af^*)P\wh{\pi}(f)\xi=0 \]
	where the last equality is a consequence of our assumption that $\wh{\pi}(f)\xi\in \fM^\perp$. This implies that $P\fM=\{0\}$, whence  $\fK=\{0\}$. Thus,
	\[ \pi(a)|_{\fM}=\rho(a)=a^{(\kappa)}, \quad a\in \A_d.\]
	Because $\xi$ is a cyclic vector for $\rho(\A_d)$, we infer from \cite[Corollary 1.10]{popescu2006} that $\kappa=1$, so that (i) holds. Lemma \ref{L:UnivStructProj} then implies that (ii) is verified.
	Moreover, 
	\[
	P_{\fM}\pi(s)|_\fM=s, \quad s\in \S_d
	\]
	which immediately implies (iii). 	
\end{proof}

Let $\tau:\O_d\to \bC$ be a bounded linear functional. Let $\pi:\O_d\to B(\fH)$ be a unital $*$-representation and let $\xi\in \fH$ be a cyclic vector for $\pi$ with $\|\xi\|=\|\tau\|^{1/2}$. Assume that the representation induced  by $\pi$ on the invariant subspace $\ol{\pi(q(\A_d))\xi}$ is unitarily equivalent to the identity representation of $\fT_d$. If $\eta\in\fH$ is a vector such that $\|\eta\|=\|\tau\|^{1/2}$ and with the property that
\[
\tau(t)=\langle \pi(t)\xi,\eta\rangle, \quad t\in \O_d
\]
then we say that the triple $(\pi,\xi,\eta)$ is a \emph{Riesz representation} for $\tau$ on $\fH$. The main result of this section is the following, which says that an analytic functional on the Cuntz algebra always admits an essentially unique Riesz representation $(\pi,\xi,\eta)$, and in addition $\pi$ is completely non-singular.

\begin{theorem}
Let $\tau:\O_d\to \bC$ be a bounded linear functional with the property that $q(\A_d)\subset \ker \tau$. Then, $\tau$ admits a completely non-singular Riesz representation $(\pi,\xi,\eta)$
such that
\[
|\tau|(t)=\langle \pi(t)\xi,\xi\rangle, \quad t\in \O_d.
\]
Furthermore, this Riesz representation enjoys the following uniqueness property: if $\pi':\O_d\to B(\fH')$ is a unital $*$-representation and $\xi',\eta'\in \fH'$ are vectors with $\|\xi'\|=\|\eta'\|=\|\tau\|^{1/2}$ such that $\xi'$ is cyclic for $\pi'$ and
\[
\tau(t)=\langle \pi'(t)\xi',\eta'\rangle, \quad t\in \O_d,
\]
then there is a unitary operator $U:\fH\to\fH'$ such that $U\xi=\xi', U\eta=\eta'$ and
\[
U\pi(t)U^*=\pi'(t), \quad t\in \O_d.
\]
	\label{T:MCEFMR}
\end{theorem}
\begin{proof}
	By Lemma \ref{L:PolarForm}, 	there exist a unital $*$-representation $\pi:\O_d\to B(\fH)$, a vector $\xi\in\fH$ with $\|\xi\|=\|\tau\|^{1/2}$ that is cyclic for $\pi$ and a partial isometry $f\in\O_d^{**}$ such that  $\wh{\pi}(f^*f)\xi=\xi$ and
	\[ \tau(t)=\ip{\pi(t)\xi,\wh{\pi}(f)\xi}, \qand |\tau|(t)=\ip{\pi(t)\xi,\xi}\]
	for every $t\in \O_d$.
 Now, the assumption that $q(\A_d)\subset \ker \tau$ implies that $\wh{\pi}(f)\xi\in (\pi(q(\A_d))\xi)^\perp$. 
	Applying Lemma \ref{L:Main}, we see that $(\pi,\xi,\eta)$ is a Riesz representation for $\tau$ on $\fH$. In turn, Proposition \ref{P:CNSrestriction} shows that $\pi$ is completely non-singular.
	
	The uniqueness statement follows at once from Lemma \ref{L:Subrep}.
%
\end{proof}

We reiterate that the previous theorem can be viewed as a genuine generalization of the classical F. and M. Riesz theorem, through the lens of Corollary \ref{C:classFM1}. The reader will notice however that in Corollary \ref{C:classFM1}, a single unitary equivalence class of  representations of $\O_1$ serves as a Riesz representation for all analytic functionals.  One might wonder whether there is such a distingushed unitary  equivalence class of representations when $d>1$. We show next that this is far from being true.

\begin{corollary}\label{C:NoOneRep}
Let $d\geq 2$. Let $\pi:\O_d\to B(\fH)$ be an irreducible, completely non-singular, unital $*$-representation. Then, there are unit vectors $\xi,\eta$ in the range of $\widehat{\pi\circ q}(I-\fq)$ such that $(\pi,\xi,\eta)$ is a Riesz representation for some bounded linear functional on $\O_d$ annihilating $q(\A_d)$. In particular,  there exist bounded linear functionals on $\O_d$ annihilating $q(\A_d)$ with unitarily inequivalent Riesz representations.
\end{corollary}
\begin{proof}
First note that $\pi$ is not singular, so that $\widehat{\pi\circ q}(\fq)\neq I$ by Lemma \ref{L:UnivStructProj}. By property (a) of  the universal structure projection $\fq$ (see Subsection \ref{SS:ncda}), we infer that the space $\fK=\widehat{\pi\circ q}(I-\fq)\fH$ is invariant for $(\pi\circ q)(\A_d)$. The corresponding representation induced  by $\pi$ is absolutely continuous by Lemma \ref{L:UnivStructProj}.  Because $d\geq 2$, by combining  \cite[Theorem 4.12]{kennedy2011} with \cite[Theorem 4.16]{kennedy2013} we see that there exists a vector $\eta\in \fK$ which is wandering for $\pi$. Let $\xi=\pi(q(L_1))\eta\in \fK$, which is still a wandering vector for $\pi$. Since we assume that $\pi$ is irreducible, $\xi$ is necessarily cyclic for $\pi$. Define $\tau:\O_d\to\bC$ as
\[
\tau(t)=\langle \pi(t)\xi,\eta\rangle, \quad t\in \O_d.
\]
Then, we find 
\[
\tau(q(L_1^*))=\langle \pi(q(L_1^*L_1))\eta,\eta \rangle=\|\eta\|^2=1
\]
and thus $\|\tau\|=1$.  The fact that $\xi$ is also wandering for $\pi$ then implies that $(\pi,\xi,\eta)$ is a Riesz representation for $\tau$. Moreover, for $a\in \A_d$ we have
\[
\tau(q(a))=\langle \pi(q(aL_1))\eta,\eta\rangle=0
\]
since $\eta$ was chosen to be wandering for $\pi$, whence $q(\fA_d)\subset \ker \tau$ as desired. The first statement has been established, and the second statement then immediately follows from it along with the uniqueness statement of Theorem \ref{T:MCEFMR}, since there exist unitarily inequivalent completely non-singular $*$-representations of $\O_d$ (see for instance Example \ref{E:AtomicD}).
\end{proof}

\section{Weak-$*$ continuous extensions of analytic functionals}\label{S:FMRnAC} 
In this section, we seek to give another generalization of the classical F. and M. Riesz theorem, in the form of Corollary \ref{C:classFM2}. In other words, we explore the existence of weak-$*$ continuous extensions for analytic functionals.  Recall that $\S_d\subset B(\fF^2_d)$ denotes the norm closed operator system generated by $\A_d$, while $\V_d\subset B(\fF^2_d)$ denotes the weak-$*$ closure of $\S_d$. We start with a positive result concerning absolute values.

\begin{theorem}
	Let $\tau:\fT_d\to \bC$ be a bounded linear functional such that $\A_d\subset \ker \tau$. 
	Then, the restriction of $|\tau|$ to $\S_d$ extends weak-$*$ continuously to $\V_d$.
	\label{T:TdAbsValAC}
\end{theorem}
\begin{proof}
By Lemma \ref{L:PolarForm}, there exist a unital $*$-representation $\pi:\fT_d\to B(\fH)$, a vector $\xi\in\fH$ with $\|\xi\|=\|\tau\|^{1/2}$ that is cyclic for $\pi$ and a partial isometry $f\in\fT_d^{**}$ such that  $\wh{\pi}(f^*f)\xi=\xi$,
		\[ \tau(t)=\ip{\pi(t)\xi,\wh{\pi}(f)\xi} \qand  |\tau| (t)=\ip{\pi(t)\xi,\xi}\]
		for every $t\in \fT_d$.  Now, the fact that $\A_d\subset \ker \tau$ implies that $\wh{\pi}(f)\xi\in (\pi(\A_d)\xi)^\perp$. 
Invoke Lemma \ref{L:Main} to see that 
\[
s\mapsto P_{\fM}\pi(s)|_{\fM}, \quad s\in \S_d
\]
 extends weak$*$-continuously to $\V_d$. Finally, observe that
	\[ |\tau|(s)=\ip{P_{\fM}\pi(s)P_{\fM}\xi,\xi}, \quad s\in\S_d, \]
	which clearly implies the desired statement.
\end{proof}

The previous result already shows that there is some amount of automatic weak-$*$ continuity associated with an analytic functional, at least as far as the absolute value is concerned. It is natural to wonder whether the functional itself enjoys such a property as well. In the classical case of measures, this is an immediate consequence of the Radon--Nikodym theorem, but such machinery is not available when working in $\S_d$ for $d>1$. The next example illustrates, perhaps surprisingly, that general analytic functionals on $\S_d$ need  not extend weak-$*$ continuously to $\V_d$, even though their absolute values do.

\begin{example}\label{E:NoACForFree}
Let $d\geq 2$. Let $\sigma:\O_d\to B(\fH)$ denote the unital $*$-representation constructed in Example \ref{E:AtomicD} with $\lambda=1,m=1$ and $\upsilon(1)=1$. Then, $\sigma$ is irreducible and of dilation type.
There is a unit vector $\eta\in \fH$ with the property that $\sigma(q(L_1))\eta=\eta$ and a wandering vector $\xi\in \fH$ for $\sigma$ such that $\eta=\sigma(q(L_2))^*\xi$. Define $\tau:\O_d\to \bC$ as
\[
\tau(t)=\langle \sigma(t)\xi,\eta\rangle, \quad t\in \fT_d.
\]
For $a\in \A_d$, we find
\[
\tau(q(a))=\langle \sigma(q(L_2a))\xi,\xi \rangle=0
\]
since $\xi$ is wandering. Hence, $\tau\circ q$ is an analytic functional, and we claim that $\tau\circ q|_{\S_d}$ does not extend weak-$*$ continuously to $\V_d$. 

To see this, for each $n\in \bN$ we  let $a_n=L_2 L_1^n\in \A_d$. Then, it is readily verified that $(a_n)$ converges to $0$ in the weak-$*$ topology of $B(\fF^2_d)$. On the other hand, we compute for every $n\in \bN$ that
\begin{align*}
(\tau \circ q)(a_n^*)&=\langle \sigma(q(L_1^{*n}L_2^*))\xi,\eta \rangle=\langle \sigma(q(L_1))^{*n} \eta,\eta\rangle=\|\eta\|^2=1
\end{align*}
so that $((\tau\circ q)(a_n^*))$ does not converge to $0$. This proves the claim.
\qed
\end{example}

Fix $d\geq 2$ along with a functional $\tau$ on $\S_d$ that annihilates $\A_d$. In the previous example, we saw that it is possible for $\tau$ to have a Riesz representation of dilation type and to fail to extend weak-$*$ continuously to $\V_d$. By Proposition \ref{P:CNSrestriction}, we know that a Riesz representation must necessarily be completely non-singular. Furthermore, if the Riesz representation of $\tau$ is absolutely continuous, then necessarily $\tau$ extends weak-$*$ continuously to $\V_d$ by \cite[Theorem 4.16]{kennedy2013} and Theorem \ref{T:rigidKennedy}.  It is therefore tempting to conjecture that the converse should hold as well, namely that if $\tau$ extends weak-$*$ continuously to $\V_d$, then its Riesz representation must be absolutely continuous. This is not the case however, as we show next.

\begin{corollary}\label{C:dilationextends}
Let $d\geq 2$ and let $\pi:\O_d\to B(\fH)$ be a completely non-singular, irreducible, unital $*$-representation. Then, there is a bounded linear functional $\tau:\O_d\to \bC$ with the following properties.
\begin{enumerate}[{\rm (a)}]
\item There are unit vectors $\xi,\eta\in \fH$ such that $(\pi,\xi,\eta)$ is a Riesz representation for $\tau$ on $\fH$.
\item The functional $\tau\circ q$ annihilates $\A_d$.
\item The restriction of $\tau\circ q$ to $\S_d$ extends weak-$*$ continuously to $\V_d$.
\end{enumerate}
\end{corollary}
\begin{proof}
By Corollary \ref{C:NoOneRep}, there are unit vectors $\xi,\eta$ in the range of $\widehat{\pi\circ q}(I-\fq)$ such that $(\pi,\xi,\eta)$ is a Riesz representation for some bounded linear functional  $\tau: \O_d\to \bC$ annihilating $q(\A_d)$. Using property (a) of the universal structure projection (see Subsection \ref{SS:ncda}), we find
\begin{align*}
(\tau\circ q)(a+b^*)&=\langle (\pi\circ q)(a+b^*)\xi,\eta\rangle\\
&=\langle \widehat{\pi\circ q}((I-\fq)(a+b^*)(I-\fq))\xi,\eta \rangle\\
&=\langle \widehat{\pi\circ q}(a(I-\fq)+(I-\fq)b^*)\xi,\eta\rangle\\
&=\widehat{\tau \circ q}(a(I-\fq)+(I-\fq)b^*)
\end{align*}
for every $a,b\in \A_d$. Invoking Lemma \ref{L:extVd}, we see that $\tau \circ q|_{\S_d}$ extends weak-$*$ continuously to $\V_d$.
\end{proof}

In light of the foregoing discussion,  it appears that there is no obvious relationship between an analytic functional on $\S_d$ extending weak-$*$ continuously to $\V_d$ and the type of its Riesz representation. Without further qualification, a perfect analogue of Corollary \ref{C:classFM2} is too much to hope for. The next development is the main result of this section, and it identifies precisely what the obstruction is in terms of the universal structure projection $\fq\in \A_d^{**}$.

\begin{theorem}
Let $\tau:\S_d\to \bC$ be a bounded linear functional with the property that $\A_d\subset \ker \tau$. 	Then $\tau$ extends weak$^*$-continuously to $\V_d$ if and only if
	\[
		\{\fq a^*-a^*\fq:a\in\A_d\}\subset \ker\widehat\tau.
	\]
	\label{T:ACComm}
\end{theorem}
\begin{proof}
We first make some preliminary remarks. Applying Lemma \ref{L:PolarForm} to a Hahn--Banach extension of $\tau$ to $\fT_d$, we obtain a unital $*$-representation $\pi:\fT_d\to B(\fH)$, a vector $\xi\in\fH$ with $\|\xi\|=\|\tau\|^{1/2}$ that is cyclic for $\pi$ and a partial isometry $f\in\fT_d^{**}$ such that  $\wh{\pi}(f^*f)\xi=\xi$, 
	\[ \tau(s)=\langle \pi(s)\xi,\widehat\pi(f)\xi\rangle, \quad s\in \S_d.\]
	Put $\eta=\widehat\pi(f)\xi$.
	Because $\A_d\subset \ker \tau$ we find $\eta\in (\pi(\A_d)\xi)^\perp$.
	Lemma \ref{L:Main} yields $\wh{\pi}(\fq)\xi=0$, whence
	\begin{equation}\label{Eq:tau1}
	 \tau(s)=\ip{\wh{\pi}( s(I-\fq) )\xi,\eta}=\wh{\tau}( s(I-\fq) ) 
	\end{equation}
	for every $s\in \S_d$. 	
	
	Let us now turn to the proof.
	Suppose first that 
	\[
	\{\fq a^*-a^*\fq:a\in\A_d\}\subset \ker\widehat\tau
	\]
	which implies
	\[
	\{(I-\fq) a^*-a^*(I-\fq):a\in\A_d\}\subset \ker\widehat\tau.
	\]
The space $\A_d^{\perp\perp}\subset \S_d^{**}$ coincides with the weak-$*$ closure of $\A_d$, so by weak-$*$ continuity of the functional $\widehat\tau$, we find  we obtain $\A_d^{\perp\perp}\subset \ker \widehat\tau$. Therefore,
	\begin{equation}\label{Eq:tau3}
	\widehat\tau(a(I-\fq))=0=\tau(a), \quad a\in \A_d
	\end{equation}
	and
	\[
	\{(I-\fq) w^*-w^*(I-\fq):w\in\A^{\perp\perp}_d\}\subset \ker \widehat\tau.
	\]
	In particular, this means that
	\begin{equation}\label{Eq:tau2}
	\widehat\tau(b^*(I-\fq))=\widehat\tau((I-\fq)b^*(I-\fq)), \quad b\in \A_d.
	\end{equation}
	For $a,b\in\A_d$, we thus find using \eqref{Eq:tau1}, \eqref{Eq:tau3} and \eqref{Eq:tau2} along with property (a) of the universal structure projection (see Subsection \ref{SS:ncda}) that
	\begin{align*}
	\tau(a+b^*)&=\tau(a)+\tau(b^*)=\widehat\tau(b^*(I-\fq))\\
	&=\widehat\tau((I-\fq)b^*(I-\fq))\\
	&=\widehat\tau(a(I-\fq))+\widehat\tau((I-\fq)b^*(I-\fq))\\
	&=\widehat\tau(a(I-\fq)+(I-\fq)b^*).
	\end{align*}
 An application of Lemma \ref{L:extVd} then shows that $\tau$ extends weak-$*$ continuously to $\V_d$.

Conversely, suppose that $\tau$ is known to extend weak-$*$ continuously to $\V_d$.
	By Goldstine's theorem, there exists a contractive net $(a_i)$ in $\A_d$ that converges to $\fq$ in the weak$^*$ topology of $\A_d^{**}$. Fix $b\in \A_d$. By properties (a) and (b) of the universal structure projection (see Subsection \ref{SS:ncda}), we see that the nets $(a_ib)$ and $(ba_i)$ converge to $0$ in the weak-$*$ topology of $B(\fF^2_d)$ for every $b\in \A_d$, and hence so do the nets $(\tau(a^*_i b^*))$ and $(\tau(b^*a^*_i))$. Thus, we find
	\[
	\widehat\tau(\fq b^*)=\lim_i \tau(a_i^* b^*)=0 \qand\widehat\tau(b^*\fq)=\lim_i \tau(b^*a_i^* )=0
	\]
	 which implies the desired property.
\end{proof}

We remark here that in the univariate context, the algebra $\A_1$ is commutative, and thus so is $\A_1^{**}$. Consequently, in this case the set $\{\fq a^*-a^*\fq:a\in\A_1\}$ is zero, and the condition of the previous theorem is automatically satisfied.

\section{Commutative algebras of multipliers}\label{S:Comm} 
Up to now, we have adopted the point of view that the non-commutative algebras $\A_d$ provide a multivariate analogue of the disc algebra for which generalizations of the classical F. and. M. Riesz theorem can be explored.  Conventional wisdom dictates that there is also a commutative side to this story, where the multivariate analogue of the disc algebra consists of multipliers on the so-called Drury--Arveson Hilbert space of functions on the complex unit ball. The purpose of this section is to examine this algebra of multipliers and corresponding extensions of the F. and M. Riesz theorem.  In fact, we will consider an entire family of spaces of functions that can be efficiently analyzed simultaneously. This family is rather rich, containing as special cases the Drury--Arveson space and the Dirichlet space. Let us now describe the setting more precisely. 

Throughout this section, we will denote by $\H$ a regular unitarily invariant reproducing kernel Hilbert space on the open unit ball $\bB_d\subset \bC^d$. The precise definition of such an object will not be needed for our purposes; for more details the reader may consult \cite{agler2002},\cite{hartz2017isom} or \cite{CH2018}. We let $\M(\H)$ denote the multiplier algebra of $\H$. It is known that the coordinate multipliers $M_{z_1},M_{z_2},\ldots,M_{z_d}$ are bounded on $\H$, so that all polynomials lie in $\M(\H)$. Accordingly, we let $\A(\H)\subset B(\H)$ be the norm closure of the polynomial multipliers, and let $\fT(\H)\subset B(\H)$ denote the $\rC^*$-algebra that $\A(\H)$ generates. The main piece of information regarding these objects that we require is the following.

\begin{theorem}\label{T:e}
Let $\H$ be a regular unitarily invariant reproducing kernel Hilbert space on $\bB_d$.
Then, there is a central projection $\fe\in \A(\H)^{**}$ with the following properties.
\begin{enumerate}[{\rm (i)}]
	\item For each $1\leq j\leq d$, we let $U_j=\fe M_{z_j}$. Then, the $d$-tuple $(U_1,\ldots,U_d)$ is a spherical unitary, in the sense that it consists of commuting normal elements satisfying $\sum_{j=1}^d U_j U_j^*=I$. 
	\item The algebra $\fN=\A(\H)^{**}\fe$ is the von Neumann algebra generated by $U_1,\ldots,U_d$. 
	\item There is a unital, completely isometric and weak-$*$ homeomorphic isomorphism $\Theta: \A(\H)^{**}(I-\fe)\to \M(\H)$ such that
	\[
	\Theta(a(I-\fe))=a, \quad a\in \A(\H).
	\]
\end{enumerate}
\end{theorem}
\begin{proof}
This is  \cite[Theorem 5.1]{DavHar2020}.
\end{proof}

The reader will notice that the central projection $\fe$ in the previous result behaves similarly to the universal structure projection $\fq\in \A_d^{**}$ (see Subsection \ref{SS:ncda}). One key difference, however, is that $\fe$ commutes with every element in $\fT(\H)^{**}$. The following standard fact relates the projection $\fe$ to weak-$*$ continuity of functionals on $\A(\H)$.

%

\begin{lemma}
	Let $\tau:\A(\H)\to \bC$ be a bounded linear functional. Then, $\tau$ extends weak-$*$ continuously to $\M(\H)$ if and only if
	\[ \tau(a)=\wh{\tau}(a(I-\fe)), \quad a\in\A(\H). \]
	\label{L:HiddenE}
\end{lemma}
\begin{proof}
	Assume first that $\tau$ extends weak-$*$ continuously to $\M(\H)$.
	By Goldstine's Theorem, there is a contractive net $(b_i)_i$ in $\A(\H)$ that converges to $I-\fe$ in the weak-$*$ topology of $\A(\H)^{**}$. By part (iii) of Theorem \ref{T:e}, we infer that $(b_i)$ converges to $1$ in the weak-$*$ topology of $B(\H)$. Hence, for $a\in \A(\H)$ we obtain
	\[
	\tau(a)=\lim_i \tau(ab_i)=\widehat\tau(a(I-\fe)).
	\]	
	Conversely, assume that \[ \tau(a)=\wh{\tau}(a(I-\fe)), \quad a\in\A(\H). \]
	Let $(b_n)_n$ be a sequence in $\A(\H)$ converging to $0$ in the weak-$*$ topology of $B(\H)$.
	Then $( b_n (I-\fe))_n$ converges to $0$ in the weak-$*$ topology of $\A(\H)^{**}$ by virtue of part (iii) of Theorem \ref{T:e}. Therefore
	\[ \lim_{n\to\infty}\tau(b_n)=\lim_{n\to\infty}\wh{\tau}( b_n(I-\fe))=0. \]
	It follows from \cite[Lemma 3.1]{BicHarMcC2018} that $\tau$ extends weak-$*$ continuously to $\M(\H)$.
\end{proof}

The next step is our crucial technical tool, and is reminiscent of Lemma \ref{L:Main}. The maximality assumption appearing therein was introduced in \cite{CTGelfand}, and it is satisfied by many spaces of interest, such as the Drury--Arveson space and the Dirichlet space; see \cite[Subsection 2.3]{CH2018} for a discussion of this notion.

\begin{lemma}
	Let $\H$ be a unitarily invariant, regular, maximal, complete Nevanlinna--Pick reproducing kernel Hilbert space on $\bB_d$. 	Let $\pi:\fT(\H)\to B(\E)$ be a unital $*$-representation. Suppose there exist a vector $\xi\in\E$ along with $f\in\fT(\H)^{**}$ such that $\wh{\pi}(f^*f)\xi=\xi$ and $\xi\in \left(\pi(\A(\H))\wh{\pi}(f)\xi\right)^\perp$. If we let $\F=\ol{\pi(\A(\H))\widehat\pi(f)\xi}$, then $\wh{\pi}(\fe)|_{\F}=0$. 
	\label{L:CommMain}
\end{lemma}
\begin{proof}
Let $\rho:\A(\H)\to B(\F)$ be the unital completely contractive homomorphism given by $\rho(a)=\pi(a)|_{\F}$ for $a\in\A(\H)$. Using Theorem \ref{T:e}, we see that the map
\[
w\fe\mapsto \widehat \rho(w\fe), \quad w\in \A(\H)^{**}
\]
is a unital completely contractive homomorphism on the von Neumann algebra $\fN=\A(\H)^{**}\fe$, and thus it is a $*$-homomorphism. In particular, the $d$-tuple $(\widehat\rho(U_1),\ldots,\widehat\rho(U_d))$ is a spherical unitary. Note also that $\wh{\rho}(\fe)$ is a central projection in the weak$^*$ closure of $\rho(\A(\H))$. 	Let $P=\wh{\rho}(\fe)\oplus 0_{\F^\bot}\in B(\H)$.
	We claim that $P$ commutes with $\pi(\fT(\H))$. 
	
	To see this, let $\G\subset \E$ denote the range of $P$, and consider the unital completely positive map $\phi:\fT(\H)\to B(\G)$ defined as
	\[
	\phi(t)=P_{\G}\pi(t)|_{\G}, \quad t\in \fT(\H).
	\]
For $a\in \A(\H)$, we see that $\phi(a)=\widehat\rho(a\fe)|_\G.$ An application of \cite[Lemma 3.6]{CTGelfand} shows that $\phi$ is a $*$-homomorphism. In turn, a classical observation of Agler \cite{agler1982} (see also \cite[Lemma 3.2]{CH2018}) reveals that $\G$ is reducing for $\pi(\fT(\H))$, so that indeed $P$ commutes with $\pi(\fT(\H))$. In particular, it commutes with $\widehat{\pi}(f)$.
%

Finally, because $\xi\in \F^\perp$ by assumption, we have that $P\xi=0$.
	For $a\in\A(\H)$, we note that
	\[ P\pi(a)\widehat\pi(f)\xi=\pi(a)\widehat{\pi}(f)P\xi=0. \]
	As $\pi(\A(\H))\widehat\pi(f)\xi$ is dense in $\F$, we have that $P=0$. 
	Therefore, $\wh{\pi}(\fe)|_{\F}=\wh{\rho}(\fe)=0$.

\end{proof}

We now arrive at the main result of this section, which is partial generalization of the classical F. and M. Riesz theorem, in the form of Corollary \ref{C:classFM2}. For a subalgebra $\B\subset B(\H)$, we use the following notation
\[
\B^\dagger=\{b^*:b\in \B\}\subset B(\H).
\]

\begin{theorem}\label{T:CommFMR}
	Let $\H$ be a unitarily invariant, regular, maximal, complete Nevanlinna--Pick reproducing kernel Hilbert space on $\bB_d$.
	Let $\tau:\fT(\H)\to \bC$ be a bounded linear functional such that $\A(\H)\subset\ker\tau$.
	Then, the restriction of $\tau$ to  $\A(\H)^\dagger$ extends weak-$*$ continuously to $\M(\H)^\dagger$.
\end{theorem}
\begin{proof}
Define a bounded linear functional $\omega:\fT(\H)\to \bC$ as
\[
\omega(t)=\ol{\tau(t^*)}, \quad t\in \fT(\H).
\]
Clearly, it suffices to prove that the restriction of $\omega$ to $\A(\H)$ extends weak-$*$ continuously to $\M(\H)$.

By Lemma \ref{L:PolarForm},  there exist a unital $*$-representation $\pi:\fT(\H)\to B(\E)$, a vector $\xi\in\E$ with $\|\xi\|=\|\tau\|^{1/2}$ that is cyclic for $\pi$ and a partial isometry $f\in\fT(\H)^{**}$ such that  $\wh{\pi}(f^*f)\xi=\xi$ and
\[
\omega(t)=\langle \pi(t)\xi,\widehat\pi(f)\xi\rangle, \quad t\in \fT(\H).
\]
By assumption, we have $\A(\H)^\dagger\subset \ker \omega$, which implies that $\xi\in (\pi(\A(\H))\widehat\pi(f)\xi)^\perp$.
	Thus, it follows from Lemma \ref{L:CommMain} that $\widehat{\pi}(\fe)\widehat\pi(f)\xi=0$. Hence, for $t\in \fT(\H)$ we observe that
	\begin{align*}
	\omega(t)&=\langle \pi(t)\xi,\widehat\pi(f)\xi \rangle=\langle \pi(t)\xi,\widehat\pi((I-\fe)f)\xi \rangle\\
	&=\langle \widehat\pi(t(I-\fe))\xi,\widehat\pi(f)\xi\rangle\\
	&=\widehat\omega(t(I-\fe))
	\end{align*} 
where we used the fact that $\widehat\pi(\fe)$ commutes with $\pi(\fT(\H))$. The desired conclusion now follows immediately from  Lemma \ref{L:HiddenE}.
\end{proof}

The previous result does not yield as satisfactory a conclusion as its non-commuta\-tive counterpart Theorem  \ref{T:ACComm}, which guarantees the existence of a weak-$*$ continuous extension at the level of operator systems. The obstacle to proving such a stronger statement in the context of unitarily invariant spaces is a current lack of a replacement for Theorems \ref{T:rigidKennedy} and \ref{T:rigidity}. Nevertheless, Theorem \ref{T:CommFMR} does yield some non-trivial information, as we illustrate next with an application.

A bounded linear functional $\phi:\fT(\H)\to \bC$ is said to be \emph{Henkin} for $\A(\H)$ if its restriction to $\A(\H)$ extends weak-$*$ continuously to $\M(\H)$. We denote the set of all such functionals by $\Hen(\A(\H))$. The set $\Hen(\A(\H)^\dagger)$ of Henkin functionals for $\A(\H)^\dagger$ is defined analogously, where instead we require the existence of a weak-$*$ continuous extension on $\M(\H)^\dagger$.
Henkin functionals have been a topic of recent interest in multivariate operator theory (see \cite{CD2016duality},\cite{CD2016abscont},\cite{BicHarMcC2018},\cite{DavHar2020} and the references therein). Nevertheless, outside of the classical setting of the ball algebra \cite[Theorem 9.6.1]{rudin2008}, the exact nature of such functionals is not fully understood. Our final result sheds some light on their structure.

\begin{corollary}\label{C:HenA=HenTA}
	Let $\H$ be a unitarily invariant, regular, maximal, complete Nevanlinna--Pick reproducing kernel Hilbert space on $\bB_d$. Then,
	\[ \mathrm{Hen}(\A(\H))=\mathrm{Hen}(\A(\H)^\dagger). \]
\end{corollary}

\begin{proof}
	Let $\phi\in\mathrm{Hen}(\A(\H))$.
	By definition, there exists a weak-$*$ continuous functional $\omega_0$ on $\M(\H)$ such that
	\[ \omega_0|_{\A(\H)}=\phi|_{\A(\H)}. \]
	Increasing the norm if necessary, we may extend $\omega_0$ to a weak-$*$ continuous linear functional $\omega:B(\H)\to \bC$. Define $\psi:\fT(\H)\to \bC$ as
	\[ \psi(t)=\omega(t)-\phi(t), \quad t\in\fT(\H). \]
	Then $\A(\H)\subset \ker\psi$, and it follows from Theorem \ref{T:CommFMR} that $\psi|_{\A(\H)^\dagger}$ extends weak-$*$ continuously to $\M(\H)^\dagger$. Clearly, the same must be true of
	\[
	\phi|_{\A(\H)^\dagger}=\omega|_{\A(\H)^\dagger}-\psi|_{\A(\H)^\dagger}
	\]
	so indeed $\phi\in\mathrm{Hen}(\A(\H)^\dagger)$. The reverse inclusion follows similarly.	
\end{proof}

We close by remarking that in the case of the ball algebra, the previous symmetry result is a direct consequence of the known characterization of Henkin functionals due to Cole and Range \cite{cole1972} (see also \cite[Theorem 9.6.1]{rudin2008}). In the absence of such a result in general, we do not know how to prove Corollary \ref{C:HenA=HenTA} directly.

\bibliographystyle{plain}
\bibliography{biblioFMRiesz.bib}

\end{document}